\documentclass[12pt,a4paper]{amsart}

\usepackage{amsmath,geometry,amssymb,amsfonts,amsthm,graphicx,enumerate,latexsym,tabularx,amscd,array}

\usepackage{refcheck} 
\norefnames
\nocitenames
\usepackage{geometry}
\usepackage{amsthm}
 \newtheorem{thm}{Theorem}[section]
 \newtheorem{prop}[thm]{Proposition}
 \newtheorem{lem}[thm]{Lemma}
 \newtheorem{cor}[thm]{Corollary}
\theoremstyle{definition}
 \newtheorem{dfn}[thm]{Definition}
\theoremstyle{remark}
 \newtheorem{rem}[thm]{Remark}
 \numberwithin{equation}{section}
\theoremstyle{definition}
\theoremstyle{remark}
 \numberwithin{equation}{section}

\renewcommand{\le}{\leqslant}\renewcommand{\leq}{\leqslant}
\renewcommand{\ge}{\geqslant}\renewcommand{\geq}{\geqslant}

\usepackage{tikz}
\usepackage[pdftex]{hyperref}
\usetikzlibrary{matrix,arrows}

\newcommand{\bbQ}{\mathbb{Q}}

\newcommand{\bbZ}{\mathbb{Z}}   

\renewcommand{\and}{\quad \mbox{and} \quad}  
\renewcommand{\le}{\leqslant}\renewcommand{\leq}{\leqslant}
\renewcommand{\ge}{\geqslant}\renewcommand{\geq}{\geqslant}

\title{Computation of the Lambda function for a finite Galois extension}


\author[Biswas]{\bfseries Sazzad Ali Biswas}

\address{
Chennai Mathematical Institute\\ 
H1, Sipcot It Park, Siruseri  \\ 
Kelambakkam, 603103\\
India}
\email{sabiswas@cmi.ac.in, sazzad.jumath@gmail.com}

\thanks{The  author is partially supported by IMU-Berlin Einstein Foundation,
Berlin, Germany and CSIR, Delhi, India }

\begin{document}

\vspace{18mm}
\setcounter{page}{1}
\thispagestyle{empty}

\begin{abstract}
By Langlands \cite{RL}, and Deligne \cite{D1} we know that the local constants are extendible functions.
Therefore, to give an explicit formula of the local constant of an induced representation of a 
local Galois group of a 
non-Archimedean local field $F$ of characteristic zero,
we have to compute the lambda function $\lambda_{K/F}$ 
for a finite extension $K/F$. In this paper, when a finite extension
$K/F$ is Galois, we give a formula for $\lambda_{K/F}$.
 
\end{abstract}

\maketitle

\section{\textbf{Introduction}}

Let $F$ be a non-Archimedean local field of characteristic zero. Let $\bar{F}$ be 
an algebraic closure of the field $F$. Consider a tower of fields $\bar{F}/K/F$ ($K/F$ is finite, but need 
not be Galois). Put $G_F:=\text{Gal}(\bar{F}/F)$, $G_K:=\text{Gal}(\bar{F}/K)$. Denote by $\rho_K$ a complex 
representation of the group $G_K$ of dimension $\text{dim}(\rho_K)$.
Langlands (cf. \cite{RL}) associated a (local) constant 
$W(\rho_F,\psi_F)$ of absolute value $1$ to every continuous finite dimensional complex 
representation $\rho_F$ of the group $G_F$. It satisfies 
\begin{equation}\label{eqn 0.1}
 W(\text{Ind}_{G_K}^{G_F}(\rho_K),\psi_F)
 =\lambda_{K/F}(\psi_F)^{\text{dim}(\rho_K)}\cdot W(\rho_K,\psi_K).
\end{equation}
Here $\psi_F$ is any nontrivial additive character of the field $F$ and 
$\psi_K:=\psi_F\circ\text{Tr}_{K/F}$, $\text{Ind}_{G_K}^{G_F}(\rho_K)$ is the representation of $G_F$ induced 
from $\rho_K$, and 
$$\lambda_{K/F}(\psi_F):=W(\text{Ind}_{G_K}^{G_F}(1_K),\psi_F),$$
where $1_K$ is the trivial representation of $G_K$. When the additive character $\psi_F$ is canonical 
(cf. \cite{JT1}, p. 92), for the local constants we simply write $W(\rho)$ instead of writing
$W(\rho,\psi_F)$ where $\rho$ is any finite dimensional complex representation of $G_F$.
The function $\lambda_{K/F}(\psi_F)$ is 
called {\bf Langlands'}
$\lambda$-function or simply $\lambda$-function.
We also can define the $\lambda$-function
via Deligne's constant $c(\rho):=\frac{W(\rho)}{W(\det(\rho))}$, where $\rho$ is a finite dimensional 
representation of $G_F$ and 
$\det(\rho)$ is the determinant of $\rho$ (see equations (\ref{eqn 3.4}), (\ref{eqn 4.4})).

In this paper, we use Langlands' convention for the local constants.
There are two other conventions for the local
constants, due to Deligne and to Bushnell $\&$ Henniart
(cf. \cite{BH}). In \cite{SABT}, Subsection 2.3.2, pp. 21-23, one can see the relations between 
the different conventions for the local constants.

The local constants were first introduced by John Tate in his 1950 Ph.D. thesis, for linear characters of 
local fields. They appear in his famous local functional equation (cf. \cite{JT2}, p. 13, equation (3.2.1))
of the local abelian $L$-function. In \cite{RL}, Langlands extended this to finite dimensional complex 
representations of a local Galois group. In 1972, Deligne (cf. \cite{D1}) showed that local constants can 
be attached to finite dimensional complex representations of local Galois groups by an elegant global method.

The local Langlands correspondence preserves local constants so the explicit computation of local 
constants is an integral part of the Langlands program. In \cite{RL}, Langlands proved that the local 
constants are extendible functions (cf. Theorem 1 on p. 105 of \cite{JT1}). 
Therefore, to compute the local constant of any induced local Galois representation, 
we have to compute the $\lambda$-function explicitly because here we need to use equation (\ref{eqn 0.1}).

Further, if we know an explicit formula for the local constants of the Galois
representations, then by using these computations we can compute global
constants, because the global constant is a product of local constants. The global
Langlands conjecture predicts that the global constant (which appears in the
functional equation for an Artin L-function) will be preserved under the global
Langlands correspondence and it should be compatible to the local Langlands correspondence at every place. 
Thus an explicit formula for the local constants will lead to
information concerning the global Langlands correspondence.

In \cite{GH}, G. Henniart has computed the $\lambda$-functions for all odd degree local extensions of a 
non-Archimedean local field of characteristic zero (cf. Proposition 2 on p. 124 of \cite{GH}). In \cite{TS},
T. Saito has shown that the Henniart's formula regarding $\lambda$-function  for odd degree extension 
is a consequence of results of J.-P. Serre (cf. \cite{JPS}) and Deligne (cf. \cite{D2}). In \cite{TS}, Saito
also has computed the $\lambda$-function for an arbitrary extension assuming the residual characteristic
of the base field is not equal to 2 (cf. Theorem on p. 508 of \cite{TS}). 
In this paper, we also compute these $\lambda$-functions for finite
Galois extensions (except the quadratic case) and our computations are more explicit than the previous results
of Saito.

Firstly, in Section 3, we compute the $\lambda$-function for odd degree Galois extension by using some 
properties of the $\lambda$-functions (cf. \cite{SABT}, Lemma 2.2.2 on p. 14) 
and Lemmas \ref{Lemma 3.3} and \ref{Lemma 4.2}, then
 we obtain the following result (cf. Theorem \ref{General Theorem for odd case}).
\begin{thm}\label{General Theorem for odd case} 
Let $E/F$ be an odd degree Galois extension of a non-Archimedean local field $F$. If 
$L\supset K\supset F $ is any finite extension inside $E$, then $\lambda_{L/K}=1$. 
\end{thm}

And in Section 4, we compute $\lambda_{1}^{G}:=W(\text{Ind}_{\{1\}}^{G}(1))$, where $G$ is a local Galois group for a
finite Galois extension.
By using Bruno Kahn's results (cf. \cite{BK}, Theorem 1) and Theorem \ref{Theorem 4.1} (due to Deligne)
we obtain the following result (cf. Theorem \ref{Theorem 4.3}).
\begin{thm}\label{Theorem 4.3}
 Let $G$ be a finite local Galois group of a non-Archimedean local field $F$. Let $S$ be a Sylow 2-subgroup of $G$. 
 Denote $c_{1}^{G}=c(\text{Ind}_{1}^{G}(1))$.
 \begin{enumerate}
 \item If $S=\{1\}$, then we have $\lambda_{1}^{G}=1$. 
  \item If the Sylow 2-subgroup $S\subset G$ is nontrivial cyclic (\textbf{exceptional case}), then
  \begin{equation*}
   \lambda_{1}^{G}=\begin{cases}
                    W(\alpha) & \text{if $|S|=2^n\ge 8$}\\
                    c_{1}^{G}\cdot W(\alpha) & \text{if $|S|\le 4$},
                   \end{cases}
\end{equation*}
where $\alpha$ is a uniquely determined quadratic 
  character of $G$.
  \item If $S$  is metacyclic but not cyclic ({\bf invariant case}), then 
  \begin{equation*}
   \lambda_{1}^{G}=\begin{cases}
                    \lambda_{1}^{V} & \text{if $G$ contains Klein's $4$ group $V$}\\
                    1 &  \text{if $G$ does not contain Klein's $4$ group $V$}.
                   \end{cases}
\end{equation*}
  \item If $S$ is nontrivial and not metacyclic, then $\lambda_{1}^{G}=1$.
 \end{enumerate}
\end{thm}
In the above theorem we observe that $\lambda_{1}^{G}=1$, except the {\bf exceptional case}  and 
the {\bf invariant case} when $G$ contains Klein's $4$-group.
Moreover, if $\alpha$ is the uniquely determined quadratic character of $G$, then $W(\alpha)=\lambda_{F_2/F}$, 
where $F_2/F\subseteq K/F$ is the 
quadratic extension corresponding to $\alpha$. In fact, in the invariant case we need to compute 
$\lambda_{1}^{V}$ where $V$ is Klein's $4$-group. If $p\ne 2$ then $V$ corresponds to a 
tame extension and in this paper we obtain  
an explicit computation of $\lambda_{1}^{V}$ in Lemma \ref{Lemma 4.6}.

Furthermore, in Appendix, we give an explicit formula 
for $\lambda_{K/F}$, where $K/F$ is a finite Galois extension of even degree 
with odd ramification index.

\section{\textbf{Notations and Preliminaries}}

Let $F$  be a non-Archimedean local field of characteristic zero, i.e., 
a finite extension of the field $\mathbb{Q}_p$ (field of $p$-adic numbers),
where $p$ is a prime.
Let $K/F$ be a finite extension of the field $F$. Let $e_{K/F}$ be the ramification index of the 
extension $K/F$ and $f_{K/F}$ be 
the residue degree of the extension $K/F$. The extension $K/F$ is called unramified 
if $e_{K/F}=1$. The extension $K/F$ is totally ramified if 
$e_{K/F}=[K:F]$. 
Let
$q_F$ be the cardinality of the residue field $k_F$ of $F$. If $\text{gcd}(p,[K:F])=1$, then the extension 
$K/F$ is called tamely ramified, otherwise wildly ramified. 
The extension $K/F$ is totally tamely ramified
if it is both totally ramified and tamely ramified.

For a tower of {\bf local} fields $K/L/F$, we have (cf. \cite{FV}, p. 39, Lemma 2.1) 
\begin{equation}
 e_{K/F}(\nu_K)=e_{K/L}(\nu_K)\cdot e_{L/F}(\nu_L),
\end{equation}
where $\nu_K$ is a valuation on $K$ and $\nu_L$ is the induced 
valuation on $L$, i.e., $\nu_L=\nu_K|_{L}$. For 
the tower of fields $K/L/F$ we simply write $e_{K/F}=e_{K/L}\cdot e_{L/F}$.
Let $O_F$ be the 
ring of integers in the local field $F$ and $P_F=\pi_F O_F$ is the unique prime ideal in $O_F$ 
and $\pi_F$ is a uniformizer, i.e., an element in $P_F$ whose valuation is one, i.e.,
 $\nu_F(\pi_F)=1$.
Let $U_F=O_F-P_F$ be the group of units in $O_F$.
Let $P_{F}^{i}=\{x\in F:\nu_F(x)\geq i\}$ and for $i\geq 0$ define $U_{F}^i=1+P_{F}^{i}$
(with proviso $U_{F}^{0}=U_F=O_{F}^{\times}$).
We also consider that $a(\chi)$ is the {\bf conductor} of 
 nontrivial character $\chi: F^\times\to \mathbb{C}^\times$, i.e., $a(\chi)$ is the smallest integer $m\geq 0$ such 
 that $\chi$ is trivial
 on $U_{F}^{m}$. We say $\chi$ is unramified if the conductor of $\chi$ is zero and otherwise ramified.
Throughout the paper when $K/F$
is unramified we choose uniformizers $\pi_K=\pi_F$. And when $K/F$ is totally ramified (both tame and wild) we choose
uniformizers $\pi_F=N_{K/F}(\pi_K)$, where $N_{K/F}$ is the norm map from $K^\times$ to $F^\times$.
In this paper, $\Delta_{K/F}$ denotes
$\det(\text{Ind}_{K/F} (1))$.

 The conductor of any nontrivial additive character $\psi$ of the field $F$ is an integer
$n(\psi)$ if $\psi$ is trivial
on $P_{F}^{-n(\psi)}$, but nontrivial on $P_{F}^{-n(\psi)-1}$.

\subsection{Local constants}

For a nontrivial multiplicative character $\chi$ of $F^\times$ and nontrivial additive character $\psi$ of $F$, we have 
\begin{equation}\label{eqn 2.4}
 W(\chi,\psi,c)=\chi(c)\frac{\int_{U_F}\chi^{-1}(x)\psi(x/c) dx}{|\int_{U_F}\chi^{-1}(x)\psi(x/c) dx|},
\end{equation}
where the Haar measure $dx$ is normalized such that the measure of $O_F$ is $1$ and 
$c\in F^\times$ with valuation $n(\psi)+a(\chi)$.
The formula (\ref{eqn 2.4}) can be modified as follows (cf. \cite{JT1}, p. 94):
\begin{equation}\label{eqn 2.5}
 W(\chi,\psi,c)=\chi(c)q_F^{-a(\chi)/2}\sum_{x\in\frac{U_F}{U_{F}^{a(\chi)}}}\chi^{-1}(x)\psi(x/c),
\end{equation}
where $c=\pi_{F}^{a(\chi)+n(\psi)}$. Now if $u\in U_F$ is a unit and if we replace $c$ by $cu$, then we would have 
\begin{equation}
W(\chi,\psi,cu)=\chi(c)q_F^{-\frac{a(\chi)}{2}}\sum_{x\in\frac{U_F}{U_{F}^{a(\chi)}}}\chi^{-1}(x/u)\psi(x/cu)=W(\chi,\psi,c).
\end{equation}
Therefore, $W(\chi,\psi,c)$ depends only on the exponent $\nu_{F}(c)=a(\chi)+n(\psi)$. Thus we 
can write $W(\chi,\psi, c)=W(\chi,\psi)$, because $c$ is determined by 
$\nu_F(c)=a(\chi)+n(\psi)$ up to a unit $u$ which has \textbf{no influence on} $W(\chi,\psi,c)$.
If $\chi$ is unramified, i.e., $a(\chi)=0$, then $\nu_F(c)=n(\psi)$. Then from the formula of $W(\chi,\psi,c)$, we can write
\begin{equation}\label{eqn 2.3.5}
 W(\chi,\psi,c)=\chi(c),
\end{equation}
and therefore, $W(1,\psi,c)=1$ if $\chi=1$ is the trivial character.

We know that this local constant satisfies the following functional equation 
(cf. \cite{JT2}):
$$W(\chi,\psi)\cdot W(\chi^{-1},\psi)=1.$$ 
This functional equation extends to 
\begin{equation}\label{eqn 2.3.23}
 W(\rho,\psi)\cdot W(\rho^{V},\psi)=\det(\rho)(-1),
\end{equation}
where $\rho$ is any virtual finite-dimensional representation of the Weil group $W_F$, $\rho^{V}$ is the 
contragredient and $\psi$ is any nontrivial additive character of $F$. 
This is formula (3) on p. 190 of \cite{BH} for $s=\frac{1}{2}$.

\subsection{\textbf{Deligne's Constants}}

Let $K/F$ be a finite Galois extension of a local field $F$ of characteristic zero. Let $G=\text{Gal}(K/F)$, and let 
$\rho: G\to\text{Aut}_{\mathbb{C}}(V)$ be a representation. Then for this representation, 
Deligne (cf. \cite{JT1}, p. 119)
defines:
\begin{equation}
 c(\rho):=\frac{W(\rho,\psi)}{W(\det(\rho),\psi)},
\end{equation}
where $\psi$ is some additive character of $F$. If we change the additive character $\psi$ to $\psi'=b\psi$, where 
$b\in F^\times$, then from \cite{BH}, p. 190,
part (2) of the Proposition, we see:
\begin{equation}\label{eqn 3.1.1}
W(\rho,b\psi)=\epsilon(\rho,\frac{1}{2},b\psi)=\det(\rho)(b)\cdot \epsilon(\rho,\frac{1}{2},\psi)
 =\det(\rho)(b)W(\rho,\psi).
\end{equation}
Also, from the property of abelian local constants we have 
$W(\det(\rho), b\psi)=\det(\rho)(b)\cdot W(\det(\rho), \psi)$, hence
\begin{center}
 $\frac{W(\rho,\ b\psi)}{W(\det(\rho),\ b\psi)}=\frac{W(\rho,\ \psi)}{W(\det(\rho),\ \psi)}=c(\rho)$.
\end{center}
This shows that the Deligne's constant $c(\rho)$ does not depend on the choice of the additive character $\psi$.

 \section{\textbf{When $K/F$ is an odd degree Galois extension}}

Let $K/F$ be a finite Galois extension of the field $F$.
It is well known (cf. \cite{BH}, Corollary 30.4 on p. 194) that the $\lambda_{K/F}$-function
is always a \textbf{fourth} root of unity.
We also have the following result due to Gallagher.
\begin{thm}[\cite{BH}, p. 188]\label{Lemma 2.5}
If $\rho$ is a (virtual) representation of $H\subset G$, then 
\begin{equation}\label{eqn 2.3}
 \det(\text{Ind}_{H}^{G}\rho)(s)=
 \Delta_{H}^{G}(s)^{\text{dim}(\rho)}\cdot (\det(\rho)\circ T_{G/H}(s)),
\end{equation}
for $s\in G$. Here $T_{G/H}$ is the transfer map from $G$ to $H$ and 
$\Delta_{H}^{G}=\det(\text{Ind}_{H}^{G}1_H)$.
\end{thm}

We assume now that the Galois groups $H\subset G$ have the fields $K\supset F$ as their base fields. Then by class field theory
we may interpret $\det(\rho)$ of equation (\ref{eqn 2.3})
as a character of $K^\times$ and $\det(\text{Ind}_{H}^{G}\rho)$ as a character of $F^\times$, and then the equation 
(\ref{eqn 2.3}) turns into an equality of two characters of $F^\times$:
\begin{equation}\label{eqn 2.12}
 \det(\text{Ind}_{H}^{G}\rho)=\Delta_{K/F}^{\text{dim}(\rho)}\cdot \det(\rho)|_{F^\times},\quad
 \text{where}\quad \Delta_{K/F}:F^\times\to\{\pm1\}
\end{equation}
is the discriminant character
 with respect to the extension $K/F$. If we consider $Z\subset H\subset G$ corresponding to the base fields
$E\supset K\supset F$ then we have 
\begin{center}
 $\Delta_{E/F}=\det(\text{Ind}_{H}^{G}(\text{Ind}_{Z}^{H}1_Z))$,
\end{center}
and with $\rho=\text{Ind}_{Z}^{H}1_Z$ we conclude from (\ref{eqn 2.12}) that 
\begin{equation}\label{eqn 2.13}
 \Delta_{E/F}=\Delta_{E/K}|_{F^\times}\cdot\Delta_{K/F}^{[E:K]}.
\end{equation}

Moreover, in terms of Deligne's constant, we can write:
\begin{equation}\label{eqn 3.4}
 \lambda_{H}^{G}:=W(\text{Ind}_{H}^{G}1_H)=
 c(\text{Ind}_{H}^{G}1_H)\cdot W(\text{det}(\text{Ind}_{H}^{G}1_H))=c_{H}^{G}\cdot W(\Delta_{H}^{G}),
\end{equation}
where $c_{H}^{G}:=c(\text{Ind}_{H}^{G}1_H)$.

Replacing Galois groups by the corresponding local fields we may write the lambda function of the finite extension $K/F$ as 
\begin{equation}\label{eqn 4.4}
 \lambda_{K/F}=c(\text{Ind}_{K/F}1)\cdot W(\Delta_{K/F}),
\end{equation}
where $c(\text{Ind}_{K/F}1)$ is Deligne's sign, and $\Delta_{K/F}$ is a quadratic character of $F^\times$ related to 
the discriminant.

From the following lemma, we can see that 
the $\lambda$-function can change by a sign if we change the additive character. 
\begin{lem}\label{Lemma 3.2}
The $\lambda$-function can change by sign if we change the additive character.
\end{lem}
\begin{proof}
Let $K/F$ be a finite separable extension of the field $F$ and let $\psi$ be a nontrivial additive character of $F$. 
We know that the local constant $W(\rho,\psi)$ is well defined for all pairs consisting of a virtual representation $\rho$ 
of the Weil group
 $W_F$ and a nontrivial additive character $\psi$ of $F$. If we change the additive character $\psi$ to $b\psi$, where $b\in F^\times$
 is a unique element for which $b\psi(x):=\psi(bx)$ for all $x\in F$, then from
 equation (\ref{eqn 3.1.1}), we have 
 \begin{equation}\label{eqn 3.13}
  W(\rho,b\psi)=\det(\rho)(b)\cdot W(\rho,\psi).
 \end{equation}
In the definition of $\lambda$-function $\rho=\text{Ind}_{K/F} 1$, therefore, by using equation (\ref{eqn 3.13}), we have
\begin{equation}\label{eqn 3.121}
 \lambda_{K/F}(b\psi)=W(\text{Ind}_{K/F} 1, b\psi)=\Delta_{K/F}(b)W(\text{Ind}_{K/F} 1,\psi)=\Delta_{K/F}(b)\lambda_{K/F}(\psi),
\end{equation}
where $\Delta_{K/F}=\det(\text{Ind}_{K/F} (1))$ is a quadratic character (a sign function),
i.e., $\Delta_{K/F}(b)\in\{\pm 1\}$. 
 
\end{proof}

\begin{lem}\label{Lemma 3.3}
Let $L/F$ be a finite Galois extension of a non-Archimedean local field $F$ which contains $K$ 
and $G=\text{Gal}(L/F)$, $H=\text{Gal}(L/K)$. 
If $H\leq G$ is a normal subgroup and if $[G:H]$ is odd, then $\Delta_{K/F}\equiv 1$ and 
$\lambda_{K/F}^{2}=1.$
\end{lem}

\begin{proof}
 If $H$ is a normal subgroup, then $\text{Ind}_{H}^{G}1_H=\text{Ind}_{\{1\}}^{G/H}1$ is the regular representation of 
 $G/H$, hence $\det(\text{Ind}_{H}^{G}1_H)=\Delta_{K/F}$ is the quadratic character of the group $G/H$.
 By the given condition order of $G/H$ is odd, then $\Delta_{K/F}\equiv 1$, hence
  $\lambda_{K/F}^{2}=\Delta_{K/F}(-1)$. Thus $\lambda_{K/F}^{2}=1$. 
\end{proof}
{\bf Note:} Since $\Delta_{K/F}\equiv 1$, $W(\Delta_{K/F})=1$. We also know that $c(\text{Ind}_{K/F}(1))\in\{\pm 1\}$.
Then from equation (\ref{eqn 4.4}) we can easily see that $\lambda_{K/F}^{2}=1$.

In the next lemma we state 
some important results for our next Theorem \ref{General Theorem for odd case}. These are the 
consequences of the following Deligne's result for the local constant of orthogonal representations.
For an orthogonal representation $\rho:G\to O(n)$,
we know a procedure how to obtain the constant $c(\rho)$ from the second Stiefel-Whitney class $s_2(\rho)$.

\begin{thm}[Deligne, \cite{JT1}, p. 129, Theorem 3]\label{Theorem 4.1}
 Let $\rho$ be an \textbf{orthogonal representation} of a finite group $G$ and let $s_2(\rho)\in H^{2}(G,\mathbb{Z}/2\mathbb{Z})$
 be the second Stiefel-Whitney class of $\rho$. 
 Let $K/F$ be a finite Galois extension of the non-Archimedean local field $F$. 
 The Galois group $G=\mathrm{Gal}(K/F)$ is a quotient group of the full Galois group 
 $G_{F}=\mathrm{Gal}(\bar{F}/F)$ which induces an inflation map
 \begin{equation}\label{eqn 32}
\mathrm{Inf}: H^{2}(G,\mathbb{Z}/2\mathbb{Z})\rightarrow H^{2}(G_{F},\mathbb{Z}/2\mathbb{Z})=\text{Br}(F)_2\cong \{\pm 1\}.
 \end{equation}
Then 
\begin{equation}
 c(\rho)=cl(s_2(\rho))\in\{\pm 1\}.
\end{equation}
Here $cl(s_2(\rho))$ denotes for the image of $s_2(\rho)$ under the composition of these 
maps (\ref{eqn 32}), and $\text{Br}(F)_2$ denotes the 2-part of the Brauer group of $F$.\\
In particular, we have $c(\rho)=1$ if $s_2(\rho)=0\in H^{2}(G,\mathbb{Z}/2\mathbb{Z})$.
\end{thm}

\begin{lem}\label{Lemma 4.2}
Let $G$ be a local Galois group.
\begin{enumerate}
 \item If $H\leq G$ is a normal subgroup of odd index $[G:H]$, then $\lambda_{H}^{G}=1$.
 \item If there exists a normal subgroup $N$ of $G$ such that $N\leq H\leq G$ and $[G:N]$ odd, then $\lambda_{H}^{G}=1$.
\end{enumerate}

\end{lem}
\begin{proof}
 \begin{enumerate}
 
\item To prove (1) we use the equation (\ref{eqn 3.4})
\begin{equation}
 \lambda_{H}^{G}=W(\mathrm{Ind}_{H}^{G}1_H)=
 c(\mathrm{Ind}_{H}^{G}1_H)\cdot W(\mathrm{det}\circ\mathrm{Ind}_{H}^{G}1_H).\label{eqn 4.7}
\end{equation}
Since $\rho=\mathrm{Ind}_{H}^{G}1_H$ is orthogonal we may compute $c(\rho)$ by using the second  Stiefel-Whitney 
class $s_2(\rho)$.
From Proposition 2(v) on p. 119 of \cite{JT1}
we know that $c(\rho)=W(\rho)/W(\det(\rho))$ is a sign. If $cl(s_2(\rho))$
is the image of $s_2(\rho)$ under inflation map (which is injective), then according to 
Deligne's Theorem \ref{Theorem 4.1}, we have:
\begin{center}
 $c(\rho)=cl(s_2(\rho))$
\end{center}
if $\rho$ is orthogonal.  
Moreover, we have 
\begin{center}
 $s_2(\mathrm{Ind}_{H}^{G}1_H)\in H^2(G/H,\mathbb{Z}/2\mathbb{Z})=\{1\}$,
\end{center}
and $W(\Delta_{H}^{G})=1$ by Lemma \ref{Lemma 3.3}.
This implies that in equation (\ref{eqn 4.7}) both factors are $=1$, hence $\lambda_{H}^{G}=1$.
\item From $N\leq H\leq G$, we obtain 
\begin{equation}
 \lambda_{N}^{G}=\lambda_{N}^{H}\cdot(\lambda_{H}^{G})^{[H:N]}.
\end{equation}
From (1) we obtain $\lambda_{N}^{G}=\lambda_{N}^{H}=1$ because $N$ is normal and the index $[G:N]$ is odd, hence 
$(\lambda_{H}^{G})^{[H:N]}=1$. Finally this implies $\lambda_{H}^{G}=1$ because $\lambda_{H}^{G}$ is a fourth root of unity
and $[H:N]$ is odd.
 \end{enumerate}

\end{proof}
\textbf{Note:} In other words, we can state this above Lemma \ref{Lemma 4.2} as follows:\\ 
Let $\Delta\subseteq G$ be a subgroup and $H\subseteq\Delta$.
Let $H'=\cap_{x\in\Delta}xHx^{-1}\subset\Delta$ be 
the largest subgroup of $H$ which is normal in $\Delta\subseteq G$. Then $\lambda_{H}^{\Delta}(W)=1$ if the index $[\Delta:H']$ is odd, in 
particular if $H$ itself is a normal subgroup of $\Delta$ of odd index.

Now we are in a position to state the main theorem for odd degree Galois extension of a non-Archimedean local field.

\begin{thm}\label{General Theorem for odd case} 
Let $E/F$ be an odd degree Galois extension of a non-Archimedean local field $F$. If 
$L\supset K\supset F $ is any finite extension inside $E$, then $\lambda_{L/K}=1$. 
\end{thm}

\begin{proof}
 By the given condition the degree of extension $[E:F]$ of $E$ over $F$ is odd. 
 Let $L$ be any 
 arbitrary intermediate field of $E/F$ which contains $K/F$. Therefore, here we have the 
 tower of fields $E\supset L\supset K\supset F$. Here the degree of extensions are all odd since $[E:F]$ is odd.
 By assumption $E/F$ is Galois, then also the extension $E/L$ and $E/K$
 are Galois and $H=\mathrm{Gal}(E/L)$ is a subgroup of $G=\mathrm{Gal}(E/K)$. 
 
 By the definition 
 we have $\lambda_{L/K}=\lambda_{H}^{G}$.
 If $H$ is a normal subgroup of $G$ then $\lambda_{H}^{G}=1$ because $|G/H|$ is odd. But $H$ need {\bf not} be a 
 normal subgroup
 of $G$ therefore, $L/K$ need not be a Galois extension.  
  Let $N$ be the \textbf{largest} normal subgroup of $G$ contained in $H$ and $N$ can be written as:
 \begin{equation*}
  N=\cap_{g\in G}gHg^{-1}.
 \end{equation*}
 Therefore, the fixed field $E^{N}$ is the \textbf{smallest normal} extension of $K$ containing $L$. Now we have 
 \begin{equation}
  \lambda_{N}^{G}=\lambda_{N}^{H}\cdot(\lambda_{H}^{G})^{[H:N]}.
 \end{equation}
 
Since $|G|$ is odd, $[G:N]$ is odd and hence $\lambda_{L/K}=1$ by Lemma \ref{Lemma 4.2}(2).
Then we may say $\lambda_{L/K}=1$ in all possible cases if $[E^{N}:K]$ is odd. When the big extension $E/F$ is odd then all intermediate
extensions will be odd. Therefore, the theorem is proved for all possible cases.

\end{proof}

\begin{rem}\label{Remark 4.4} 
{\bf (1).} If the Galois extension $E/F$ is infinite then we say it is \textbf{odd} if 
 $[K:F]$ is odd for all sub-extensions of finite degree. This means the pro-finite group $\mathrm{Gal}(E/F)$ can 
 be realized as the projective limit of finite groups which are all of odd order.
 If $E/F$ is a Galois extension of odd order in this more general sense, then again we will have 
 $\lambda_{L/K}=1$ in all cases where $\lambda$-function is defined.\\
{\bf (2).}
But this above Theorem \ref{General Theorem for odd case} is not true if $K/F$ is not {\bf Galois}. 
Guy Henniart gives {\bf``An amusing formula"} 
 (cf. \cite{GH}, p. 124, Proposition 2) for 
 $\lambda_{K/F}$, when $K/F$ is arbitrary odd degree extension, and this formula is:
 \begin{equation}
  \lambda_{K/F}=W(\Delta_{K/F})^{n}\cdot\left(\frac{2}{q_F}\right)^{a(\Delta_{K/F})},
 \end{equation}
where $K/F$ is an extension in $\bar{F}$ with finite odd degree $n$, and $\left(\frac{2}{q_F}\right)$ is the Legendre symbol if 
$p$ is odd and is $1$ if $p=2$. Here $a$ denotes the exponent of the Artin-conductor. 

\end{rem}

\section{\textbf{Computation of $\lambda_{1}^{G}$ where $G$ is a finite local Galois group}}

From equation (\ref{eqn 3.4}),
we observe that to compute $\lambda_{H}^{G}$ we need to compute the Deligne's constant $c_{H}^{G}$
and $W(\Delta_{H}^{G})$.

In this section, we need the following Gallagher's result. 
\begin{thm}[Gallagher, \cite{GK}, Theorem $30.1.8$]\label{Theorem 1.3}
 Assume that $H$ is a normal subgroup of $G$, hence $\Delta_{H}^{G}=\Delta_{1}^{G/H}$, then 
 \begin{enumerate}
  \item $\Delta_{H}^{G}=1_G$, where $1_G$ is the trivial representation of $G$, unless the Sylow $2$-subgroups of $G/H$ are cyclic and 
  nontrivial.
  \item If the Sylow $2$-subgroups of $G/H$ are cyclic and nontrivial, then $\Delta_{H}^{G}$ is the only linear character of $G$
  of order $2$.
 \end{enumerate}
\end{thm}

\begin{dfn}[\textbf{$2$-rank of a finite abelian group}]
Let $G$ be a finite abelian group. Then from the 
elementary divisor theorem for finite abelian groups,  we can write
\begin{equation}
 G\cong\bbZ/{m_1}\bbZ\times\bbZ/{m_2}\bbZ\times\cdots\times\bbZ/{m_s}\bbZ,
\end{equation}
where $m_1|m_2|\cdots|m_s$ and $\prod_{i=1}^{s}m_i=|G|$. We define 
\begin{center}
the $2$-rank of $G:=$the number of $m_i$-s which are even
\end{center}
 and we set 
\begin{center}
 $\mathrm{rk}_{2}(G)=$ $2$-rank of $G$.
\end{center}
When the order of an abelian group $G$ is odd, from the structure of $G$ we have $\mathrm{rk}_2(G)=0$, i.e.,
there are no 
even $m_i$-s for $G$.

\end{dfn}

\begin{rem}[\textbf{Remark on Theorem \ref{Theorem 1.3}}]\label{Remark 3.4}
If $G$ is a finite group with subgroups $H'\subset H\subset G$, then for
$\Delta_{H}^{G}=\det(\text{Ind}_{H}^{G}1_H)$ we know from  
Gallagher's Theorem \ref{Lemma 2.5}
\begin{align}
 \Delta_{H'}^{G}
 &=\det(\text{Ind}_{H'}^{G}1_{H'})=\det(\text{Ind}_{H}^{G}(\text{Ind}_{H'}^{H}1_{H'}))\nonumber\\
 &=(\Delta_{H}^{G})^{[H:H']}\cdot\det(\text{Ind}_{H'}^{H}1_{H'})\circ T_{G/H}\nonumber\\
 &=(\Delta_{H}^{G})^{[H:H']}\cdot(\Delta_{H'}^{H}\circ T_{G/H}).\label{eqn 3.28}
\end{align}
Now we use equation (\ref{eqn 3.28}) for $H'=\{1\}$ and $H=[G,G]=G'$. Then we have 
\begin{equation}\label{eqn 3.29}
 \Delta_{1}^{G}=(\Delta_{G'}^{G})^{|G'|}\cdot\Delta_{1}^{G'}\circ T_{G/G'}=(\Delta_{G'}^{G})^{|G'|},
\end{equation}
because by Theorem 10.25 on p. 320 of \cite{IM},  $T_{G/G'}$ is the trivial map.



We also know that $G'$ is a normal subgroup of $G$, then we can write $\text{Ind}_{G'}^{G}1_{G'}\cong\text{Ind}_{1}^{G/G'}1$, hence 
$\Delta_{G'}^{G}=\Delta_{1}^{G/G'}$. So we have 
\begin{equation}\label{eqn 3.30}
 \Delta_{1}^{G}=(\Delta_{G'}^{G})^{|G'|}=(\Delta_{1}^{G/G'})^{|G'|}.
\end{equation}
From the above equation (\ref{eqn 3.30}) we observe that $\Delta_{1}^{G}$ always reduces to the \textbf{abelian case} because 
$G/G'$ is abelian. Moreover, we know that:\\
\emph{If $G$ is abelian then $\text{Ind}_{1}^{G}1=r_G$ is the sum of all characters of $G$, hence from Miller's result 
(cf. \cite{PC}, Theorem 6) for the abelian
group $\widehat{G}$ we obtain:
\begin{align}
 \Delta_{1}^{G}
 &=\det(\text{Ind}_{1}^{G}1)=\det(\sum_{\chi\in\widehat{G}}\chi)\nonumber\\
 &=\prod_{\chi\in\widehat{G}}\det(\chi)=\prod_{\chi\in\widehat{G}}\chi\nonumber\\
 &=\begin{cases}
    \alpha & \text{if $\text{rk}_2(G)=1$}\\
    1 & \text{if $\text{rk}_2(G)\ne 1$},
   \end{cases}\label{eqn 3.31}
\end{align}
where $\alpha$ is the uniquely determined quadratic character of $G$.}
\end{rem}


Moreover, since $G/G'$ is abelian, by using equation (\ref{eqn 3.31}) for $G/G'$, from equation (\ref{eqn 3.30}) we obtain:

\begin{lem}\label{Lemma 3.41}
 Let $G$ be a finite group and let $S$ be a Sylow 2-subgroup of $G$. Then the following are equivalent:
 \begin{enumerate}
  \item $S<G$ is nontrivial cyclic;
  \item $\Delta_{1}^{G}\ne 1$, is the unique quadratic character of $G$;
  \item $\text{rk}_2(G/G')=1$ and $|G'|$ is odd.
 \end{enumerate}

\end{lem}

\begin{proof}
 Take $H=\{1\}$ in Gallagher's Theorem \ref{Theorem 1.3} and we can see that $(1)$ and $(2)$ are equivalent. From  
 equation (\ref{eqn 3.30}) we can see $(2)$ implies the condition $(3)$.
 
Now we are left to show that $(3)$ implies $(1)$. Let $S'$ be a Sylow 2-subgroup of $G/G'$. Since $\text{rk}_2(G/G')=1$, hence 
$\text{rk}_2(S')=1$, and therefore, $S'$ is cyclic. Moreover, 
 $|G'|$ is odd, hence $|S|=|S'|$. Let $f:G\to G/G'$ be the canonical group homomorphism. 
 Since $|G'|$ is odd, 
 and $\text{rk}_2(G/G')=1$, $f|_{S}$ is an isomorphism from $S$ to $S'$. Hence $S$ is a nontrivial cyclic Sylow $2$-subgroup of $G$.

This completes the proof. 

\end{proof}

\begin{thm}[\textbf{Schur-Zassenhaus}]\label{Theorem 3.61}
  If $H\subset G$ is a normal subgroup such that 
 $|H|$ and $[G:H]$ are relatively prime, then $H$ will have a complement $S$ that is a subgroup of $G$ such that 
 \begin{center}
  $G=H\rtimes S$
 \end{center}
is a semidirect product.
\end{thm}
Let $G$ be a local Galois group. Since $G$ is solvable, $G$ has Hall-subgroups 
$H\subset G$ of all types 
such that $[G:H]$ and $|H|$ are relatively prime. In particular, $G$ will have an odd Hall subgroup $H\subset G$ such that $|H|$ 
is odd and $[G:H]$ is power of $2$. 
From this we conclude the following proposition.
\begin{prop}\label{Proposition 3.5}
Let $G$ be a finite local Galois group of a non-Archimedean local field.
 Let $H\subset G$ be an odd order Hall subgroup of $G$ (which is unique up to conjugation). Then we have 
 \begin{equation}\label{eqn 3.11}
  \lambda_{1}^{G}=(\lambda_{H}^{G})^{|H|}.
 \end{equation}
Hence $\lambda_{1}^{G}=\lambda_{H}^{G}$ if $|H|\equiv 1\pmod{4}$ and $\lambda_{1}^{G}=(\lambda_{H}^{G})^{-1}$ if 
$|H|\equiv 3\pmod{4}$.\\
If the local base field $F/\bbQ_p$ has residue characteristic $p\ne 2$, then the odd order Hall subgroup $H\subset G$ is a normal subgroup
and therefore, $\lambda_{H}^{G}=\lambda_{1}^{G/H}$, where $G/H\cong S$ is isomorphic to a Sylow $2$-subgroup of $G$. For 
$G=\text{Gal}(E/F)$
this means that we have a unique normal extension $K/F$ in $E$ such that $\text{Gal}(K/F)$ is isomorphic to a Sylow $2$-subgroup
of $G$, and we will have 
\begin{center}
 $\lambda_{E/F}=\lambda_{K/F}^{[E:K]}$.
\end{center}
\end{prop}

\begin{proof}
We know that the local Galois group $G$ is solvable, then $G$ has an odd order Hall subgroup $H\subset G$.
 Then the formula (\ref{eqn 3.11}) follows because $\lambda_{1}^{H}=1$ (here $|H|$ is 
 odd and $H$ is a subgroup of the local Galois group $G$).
 
 Let now $p\ne 2$ and let $H$ be an odd order Hall subgroup of $G$. The 
 ramification subgroup $G_1\subset G$ is a normal subgroup of order a power of $p$, hence $G_1\subset H$, and $H/G_1\subset G/G_1$
 will be an odd order Hall subgroup of $G/G_1$. But the group $G/G_1$ is \textbf{supersolvable}. 
 We also know 
 that the odd order Hall subgroup of a supersolvable group is normal. Therefore, $H/G_1$ is normal in $G/G_1$, and this implies that 
 $H$ is normal in $G$. Now we can use Theorem \ref{Theorem 3.61} and we obtain $G/H\cong S$ where
 $S$ must be a Sylow $2$-subgroup. Therefore, when $p\ne 2$ we have 
 \begin{center}
  $\lambda_{1}^{G}=\lambda_{E/F}=(\lambda_{1}^{G/H})^{|H|}=\lambda_{K/F}^{[E:K]}$,
 \end{center}
where $G=\text{Gal}(E/F)$, $H=\text{Gal}(E/K)$ and $G/H=\text{Gal}(K/F)\cong S$.
 
 \end{proof}

Let $F/\bbQ_p$ be a local field with $p\ne 2$. Let $K/F$ be the extension such 
that $\text{Gal}(K/F)=V$ Klein's $4$-group.
In the following lemma we give an explicit formula for the $\lambda_{1}^{V}=\lambda_{K/F}$.

\begin{lem}\label{Lemma 4.6}
Let $F/\bbQ_p$ be a local field with $p\ne 2$. Let $K/F$ be the uniquely determined extension with $V=\text{Gal}(K/F)$, Klein's $4$-group. 
Then \\
 $\lambda_{1}^{V}=\lambda_{K/F}=-1$ if $-1\in F^\times$ is a square, i.e., $q_F\equiv 1\pmod{4}$, and\\
 $\lambda_{1}^{V}=\lambda_{K/F}=1$ if $-1\in F^\times$ is not a square , i.e., if $q_F\equiv 3\pmod{4}$,\\
 where $q_F$ is the cardinality of the residue field of $F$.
\end{lem}

\begin{proof}
If $p\ne 2$ then the square class group $F^\times/{F^\times}^2$ is Klein's 4-group, and $K/F$
is the unique abelian extension such that $N_{K/F}(K^\times)={F^\times}^2$, hence 
$$\text{Gal}(K/F)\cong F^\times/{F^\times}^2=V.$$
Since $V$ is abelian, we can write $\widehat{V}\cong V$. This implies that
there are exactly three nontrivial characters of $V$ and they are quadratic. 
By class field theory we can consider them as quadratic characters of $F^\times$.
Each of these quadratic characters determines a quadratic extension of $F$. 
Thus there are three quadratic subextensions $L_i/F$ in $K/F$, where $i=1,2,3$. We denote $L_1/F$ the 
unramified extension whereas $L_2/F$ and $L_3/F$ are tamely ramified. Then we can write 
\begin{equation}
 \lambda_{K/F}=\lambda_{K/L_i}\cdot\lambda_{L_i/F}^{2} 
\end{equation}
for all $i\in\{1,2,3\}$.
 The group $V$ has four characters $\chi_i$, $i=0,\cdots, 3$, where $\chi_0\equiv 1$ and 
 $\chi_i(i=1,2,3)$ are three characters of $V$ such that $\mathrm{Gal}(K/L_i)$ is the kernel of $\chi_i$, in other
words, $\chi_i$ is the quadratic character of $F^\times/N_{L_i/F}(L_{i}^{\times})$.

Let $r_V=\mathrm{Ind}_{\{1\}}^{V} 1$,
 then 
 \begin{center}
  $\Delta_{1}^{V}=\det(r_V)=\prod_{i=0}^{3}\chi_i\equiv 1$,
 \end{center}
because $\chi_3=\chi_1\cdot\chi_2$. Therefore, $W(\Delta_{1}^{V})=1$ and 
\begin{center}
 $\lambda_{K/F}=c(r_V)$
\end{center}
is Deligne's constant. More precisely, we have 
\begin{equation}\label{eqn 4.26}
 \lambda_{K/F}=W(\chi_1)\cdot W(\chi_2)\cdot W(\chi_1\chi_2).
\end{equation}
But here $\chi_1$ is unramified and therefore, $W(\chi_1)=\chi_1(c_1)$ (see equation (\ref{eqn 2.3.5})) 
and by using unramified character twisting formula, $W(\chi_1\chi_2)=\chi_1(c_2)\cdot W(\chi_2)$, where 
$c_2=\pi_F c_1$ because $a(\chi_2)=1+a(\chi_1)=1$. Therefore, the equation (\ref{eqn 4.26}) implies:
\begin{equation}
 \lambda_{K/F}=\chi_1(c_1)^{2}\cdot\chi_1(\pi_F)\cdot W(\chi_2)^{2}=-\chi_2(-1),
\end{equation}
since $\chi_1(\pi_F)=-1$.
Similarly, putting $\chi_2=\chi_{1}^{-1}\chi_3=\chi_1\chi_3$ and $\chi_1\chi_2=\chi_3$ in the equation (\ref{eqn 4.26}) we have 
\begin{equation}
 \lambda_{K/F}=-\chi_3(-1).
\end{equation}
Therefore, we have $\lambda_{K/F}=-\chi_i(-1)$ for $i=2,3$.

Moreover, we know that 
\begin{equation*}
 \chi_i(-1)=\begin{cases}
             1 & \text{if $-1\in F^\times$ is a square, i.e., $q_F\equiv 1\pmod{4}$}\\
             -1 & \text{if $-1\in F^\times$ is not a square, i.e., $q_F\equiv 3 \pmod{4}$}.\\
            \end{cases}
\end{equation*}
Thus finally we conclude that 
\begin{equation*}
 \lambda_{K/F}=-\chi_i(-1)=\begin{cases}
             -1 & \text{if $-1\in F^\times$ is a square, i.e., $q_F\equiv 1\pmod{4}$}\\
             1 & \text{if $-1\in F^\times$ is not a square, i.e., $q_F\equiv 3 \pmod{4}$}.\\
            \end{cases}
\end{equation*}

\end{proof}
For proving our next theorem we need 
the following theorem due to Bruno Kahn.

\begin{thm}[\cite{BK}, S\'{e}rie 1-313, Theorem 1]\label{Theorem 4.2}
 Let $G$ be a finite group, $r_G$ its regular representation. Let $S$ be any $2$-Sylow subgroup of $G$. Then $s_2(r_G)=0$
 in the following cases:
 \begin{enumerate}
  \item $S$ is a cyclic group of order $\geq 8$;
  \item $S$ is a generalized quaternion group;
  \item $S$ is not a metacyclic group.
 \end{enumerate}
\end{thm}
In the following theorem we give a general formula for $\lambda_{1}^{G}$, where $G$ is a finite local Galois group. 

\begin{thm}\label{Theorem 4.3}
 Let $G$ be a finite local Galois group of a non-Archimedean local field $F$. Let $S$ be a Sylow 2-subgroup of $G$.
 \begin{enumerate}
 \item If $S=\{1\}$, then we have $\lambda_{1}^{G}=1$. 
  \item If the Sylow 2-subgroup $S\subset G$ is nontrivial cyclic (\textbf{exceptional case}), then
  \begin{equation}
   \lambda_{1}^{G}=\begin{cases}
                    W(\alpha) & \text{if $|S|=2^n\ge 8$}\\
                    c_{1}^{G}\cdot W(\alpha) & \text{if $|S|\le 4$},
                   \end{cases}
\end{equation}
where $\alpha$ is a uniquely determined quadratic 
  character of $G$.
  \item If $S$  is metacyclic but not cyclic ({\bf invariant case}), then 
  \begin{equation}
   \lambda_{1}^{G}=\begin{cases}
                    \lambda_{1}^{V} & \text{if $G$ contains Klein's $4$ group $V$}\\
                    1 &  \text{if $G$ does not contain Klein's $4$ group $V$}.
                   \end{cases}
\end{equation}
  \item If $S$ is nontrivial and not metacyclic, then $\lambda_{1}^{G}=1$.
 \end{enumerate}
\end{thm}

\begin{proof}
 {\bf (1).} When $S=\{1\}$, i.e., $|G|$ is odd, we know from Theorem \ref{General Theorem for odd case} that $\lambda_{1}^{G}=1$.\\
 {\bf (2).} When $S=<g>$ is a nontrivial cyclic subgroup of $G$, $\Delta_{1}^{G}$ is nontrivial
 (because $\Delta_{1}^{G}(g)=(-1)^{|G|-\frac{|G|}{|S|}}=-1$) and 
 by Lemma \ref{Lemma 3.41}, $\Delta_{1}^{G}=\alpha$,
 where $\alpha$ is a uniquely determined quadratic character of $G$. Then we obtain
 \begin{center}
  $\lambda_{1}^{G}=c_{1}^{G}\cdot W(\Delta_{1}^{G})=c_{1}^{G}\cdot W(\alpha)$.
 \end{center}
 If $S$ is cyclic of order $2^n\ge 8$, then by Theorem \ref{Theorem 4.2} (case 1) and Theorem \ref{Theorem 4.1} 
 we have $c_{1}^{G}=1$, hence $\lambda_{1}^{G}=W(\alpha)$. \\
{\bf (3).} When the Sylow 2-subgroup $S\subset G$ is metacyclic but not cyclic (invariant case):\\
If $G$ contains Klein's $4$-group $V$, then $V\subset S$ because all Sylow $2$-subgroups are conjugate to each other. Then we have
$V<S<G$.
So from the properties of the $\lambda$-function, we have 
\begin{center}
 $\lambda_{1}^{G}=\lambda_{1}^{V}\cdot(\lambda_{V}^{G})^{4}=\lambda_{1}^{V}$.
\end{center}

Now assume that $G$ does not contain Klein's $4$-group. Then by assumption $S$ is metacyclic, 
not cyclic and does not contain Klein's $4$-group. We are going to see that this
implies: $S$ is generalized quaternion, and therefore, by Theorem \ref{Theorem 4.2}, $s_2(\text{Ind}_{1}^{G}(1))=0$, hence
$c_{1}^{G}=1$.

We use the following criterion for generalized quaternion groups: A finite
$p$-group in which there is a unique subgroup of order $p$ is either cyclic or generalized
quaternion (cf. \cite{MH}, p. 189, Theorem 12.5.2).\\
So it is enough to show: If $S$ does not contain Klein's $4$-group then $S$ has precisely one
subgroup of order $2$. We consider the center $Z(S)$ which is a nontrivial abelian
$2$-group. If it would be non-cyclic then $Z(S)$, hence $S$ would contain Klein's $4$-group. So
$Z(S)$ must be cyclic, hence we have precisely one subgroup $Z_2$ of order $2$ which sits in
the center of $S$. Now assume that $S$ has any other subgroup $U\subset S$ which is of order $2$.
Then $Z_2$ and $U$ would generate a Klein-$4$-group in $S$ which by our assumption cannot
exist. Therefore, $Z_2\subset S$ is the only subgroup of order $2$ in $S$. But $S$ is not cyclic, so it is
generalized quaternion.\\
Thus we can write $\lambda_{1}^{G}=c_{1}^{G}\cdot W(\Delta_{1}^{G})=W(\Delta_{1}^{G})$. Now to complete the proof we need to show 
that $W(\Delta_{1}^{G})=1$. This follows from Lemma \ref{Lemma 3.41}.\\
{\bf (4).} When the Sylow 2-subgroup $S$ is nontrivial and not metacyclic.\\
 We know that every cyclic group is also a metacyclic group. Therefore, when $S$ is nontrivial and not metacyclic, we are \textbf{not}
 in the position: $\text{rk}_2(G/G')=1$ and $|G'|$ is odd.
 This gives $\Delta_{1}^{G}=1$, hence $W(\Delta_{1}^{G})=1$. Furthermore,
by using the Theorem \ref{Theorem 4.2} and Theorem \ref{Theorem 4.1} we obtain the
 second Stiefel-Whitney class $s_2(\text{Ind}_{1}^{G}(1))=0$, hence $\lambda_{1}^{G}=c_{1}^{G}\cdot W(\Delta_{1}^{G})=1$.

This completes the proof.
 
\end{proof}

In the above Theorem \ref{Theorem 4.3} we observe that if we are in the \textbf{Case 3}, then by using Lemma \ref{Lemma 4.6} we can 
give complete formula of $\lambda_{1}^{G}$ for $p\ne 2$. 
Moreover, by using Proposition \ref{Proposition 3.5} in \textbf{case 2}, we boil down to 
the computation of $\lambda_{K/F}$, where $K/F$ is quadratic.

\begin{cor}\label{Lemma 3.10}
 Let $G=\text{Gal}(E/F)$ be a finite local Galois group of a non-Archimedean local field $F/\bbQ_p$ with $p\ne 2$. 
 Let $S\cong G/H$ be a nontrivial Sylow 2-subgroup of $G$, where $H$ is a uniquely determined Hall subgroup of odd order. Suppose that 
 we have a tower $E/K/F$
 of fields such that $S\cong \text{Gal}(K/F)$, $H=\text{Gal}(E/K)$ and $G=\text{Gal}(E/F)$. Let $\alpha$ be the uniquely
 determined quadratic character of $G$.
 \begin{enumerate}
  \item If $S\subset G$ is cyclic, then 
  \begin{enumerate}
 \item
 \begin{equation*}
 \lambda_{1}^{G}=\lambda_{K/F}^{\pm 1}=\begin{cases}
                  \lambda_{K/F}=W(\alpha) & \text{if $[E:K]\equiv 1\pmod{4}$}\\
                  \lambda_{K/F}^{-1}=W(\alpha)^{-1} & \text{if $[E:K]\equiv -1\pmod{4}$}.
                 \end{cases}
 \end{equation*}
 \item 
 \begin{equation*}
\lambda_{1}^{G}=\beta(-1)W(\alpha)^{\pm 1}=\beta(-1)\times\begin{cases}
                  W(\alpha) & \text{if $[E:K]\equiv 1\pmod{4}$}\\
                  W(\alpha)^{-1} & \text{if $[E:K]\equiv -1\pmod{4}$}
                 \end{cases}
 \end{equation*}
if $K/F$ is cyclic of order $4$ with generating character $\beta$ such that 
 $\beta^2=\alpha=\Delta_{K/F}$.
 \item 
 \begin{equation*}
  \lambda_{1}^{G}=\lambda_{K/F}^{\pm 1}=\begin{cases}
                  \lambda_{K/F}=W(\alpha) & \text{if $[E:K]\equiv 1\pmod{4}$}\\
                  \lambda_{K/F}^{-1}=W(\alpha)^{-1} & \text{if $[E:K]\equiv -1\pmod{4}$}
                 \end{cases}
 \end{equation*}
 if $K/F$ is cyclic of order $2^n\ge 8$.
\end{enumerate}
  And if the $4$th roots of unity are in $F$, we have 
$$\lambda_{1}^{G}=\lambda_{K/F}.$$ 
\item If $S$ is metacyclic but not cyclic and the $4$th roots of unity are in $F$, then 
\begin{enumerate}
 \item $\lambda_{1}^{G}=-1$ if $V\subset G$,
 \item $\lambda_{1}^{G}=1$ if $V\not\subset G$.
\end{enumerate}
\item The {\bf Case 4} of Theorem \ref{Theorem 4.3}  will not occur in this case.

\end{enumerate}

\end{cor}
 
\begin{proof}
{\bf (1).}
In the case when $p\ne 2$ we know from Proposition \ref{Proposition 3.5} that the odd Hall-subgroup $H<G$ is actually a normal subgroup
with quotient $G/H\cong S$. So if $G=\text{Gal}(E/F)$ and $K/F$ is the maximal $2$-extension inside $E$ then $\text{Gal}(K/F)=G/H\cong S$.
And we obtain:
\begin{equation}\label{eqn 3.8}
 \lambda_{1}^{G}=(\lambda_{1}^{G/H})^{|H|}=\begin{cases}
                                            \lambda_{K/F} & \text{if $[E:K]=|H|\equiv 1\pmod{4}$}\\
                                            \lambda_{K/F}^{-1} & \text{if $[E:K]=|H|\equiv -1\pmod{4}$}.
                                           \end{cases}
\end{equation}
So it is enough to compute $\lambda_{K/F}$ for $\text{Gal}(K/F)\cong S$, i.e., we can reduce the computation to the case where $G=S$.

We know that $\lambda_{K/F}=W(\text{Ind}_{K/F}(1))=\prod_{\chi}W(\chi)$, where $\chi$ runs over all characters of the cyclic group 
$\text{Gal}(K/F)$. 
If $[K:F]=2$ then $\text{Ind}_{K/F}(1)=1+\alpha$, where $\alpha$ 
is a quadratic character of $F$ associated to $K$ by class field theory, hence $\alpha=\Delta_{K/F}$.
Thus $\lambda_{K/F}=W(\alpha)$.

If $[K:F]=4$ then $\text{Ind}_{K/F}(1)=1+\beta+\beta^2+\beta^3$, where $\beta^2=\alpha=\Delta_{K/F}$ and 
$\beta^3=\beta^{-1}$, hence by the functional equation of local constant we have:
\begin{center}
 $W(\beta)W(\beta^{-1})=\beta(-1)$.
\end{center}
We then obtain:
\begin{center}
 $\lambda_{K/F}=W(\text{Ind}_{K/F}(1))=W(\beta)W(\beta^2)W(\beta^3)=\beta(-1)\times W(\alpha)$.
\end{center}
If $S$ is cyclic of order $2^n\ge 8$, then by using Theorem \ref{Theorem 4.2} in Theorem \ref{Theorem 4.1}, 
we have $c_{1}^{S}=1$. Again from equation (\ref{eqn 3.31}) 
we have $W(\Delta_{1}^{S})=W(\alpha)$ because $\text{rk}_2(S)=1$,
where $\alpha$ is the uniquely determined quadratic character of $F$. Thus we obtain
$$\lambda_{K/F}=c_{1}^{S}\cdot W(\Delta_{1}^{S})=W(\alpha).$$

Finally by using the equation (\ref{eqn 3.8}) we obtain our desired results.

Now we denote $i=\sqrt{-1}$ and consider it in the algebraic closure of $F$. If $i\not\in F$ then $p\ne 2$ implies that 
$F(i)/F$ is the unramified extension of degree $2$.

Now assume that $i\in F$.
Then first of all we know that $\lambda_{H}^{G}$ is always is a {\bf sign} because 
\begin{center}
 $(\lambda_{H}^{G})^2=\Delta_{H}^{G}(-1)=\Delta_{H}^{G}(i^2)=1$.
\end{center}
Then the formula (\ref{eqn 3.8}) turns into 
\begin{center}
 $\lambda_{1}^{G}=(\lambda_{1}^{G/H})^{|H|}=\lambda_{1}^{G/H}$,
\end{center}
where $G/H=\text{Gal}(K/F)\cong S$. Therefore, in {\bf Case 2} of Theorem \ref{Theorem 4.3} we have now same formulas as 
above but with $1$ instead of $\pm 1$.

{\bf (2).}
Moreover, when $p\ne 2$
we know that always $\lambda_{1}^{V}=-1$ if $i\in F$ (cf. Lemma \ref{Lemma 4.6}). Again, if $V\subseteq S$, hence $V\subseteq G$, and
we have 
$$\lambda_{1}^{G}=\lambda_{1}^{V}\cdot(\lambda_{V}^{G})^4=\lambda_{1}^{V}.$$
Therefore, when $S$ is metacyclic but not cyclic we can simply say:\\
$\lambda_{1}^{G}=\lambda_{1}^{V}=-1$, if $V\subset G$,\\
and if $V\not\subset G$, then from Theorem \ref{Theorem 4.3}(3), we can conclude $\lambda_{1}^{G}=1$.

{\bf (3).}
If the base field $F$ is $p$-adic with $p\ne 2$, then as a Galois group $S$ corresponds to
a tamely ramified extension (because the degree $2^n$ is prime to $p$), and therefore, $S$ must
be metacyclic. Therefore, the \textbf{Case 4} of Theorem \ref{Theorem 4.3} can never occur if $p\ne 2$.

\end{proof}


\begin{rem}
If $S$ is cyclic of order $2^n\ge 8$, then we have two formulas:\\
$\lambda_{1}^{G}=W(\alpha)$ as obtained in Theorem \ref{Theorem 4.3}(2), and $\lambda_{1}^{G}=W(\alpha)^{\pm 1}$ in 
Corollary \ref{Lemma 3.10}. So we observe that for 
$|S|=2^n\ge 8$ and $|H|\equiv -1\pmod{4}$ the value of $W(\alpha)$ must be a sign for $p\ne 2$.

In {\bf Case 3} of Theorem \ref{Theorem 4.3} we notice that $\Delta_{1}^{G}\equiv 1$, hence $\lambda_{1}^{G}=c_{1}^{G}$.
We know also that this Deligne's constant $c_{1}^{G}$ takes values $\pm 1$ 
(cf. Proposition 2(v) on p. 119 of \cite{JT1}). 
Moreover, we also notice that the Deligne's constant of a representation is independent of the choice of the additive character.
Therefore, in Case 3 of Theorem \ref{Theorem 4.3}, 
$\lambda_{1}^{G}=c_{1}^{G}\in\{\pm 1\}$ will {\bf not} depend on the choice of the additive character. Since in Case 3 the 
computation of $\lambda_{1}^{G}$ does not depend on the choice of the additive character, we call this case the 
{\bf invariant case}.

Furthermore, in \cite{BK2}, Bruno Kahn deals with $s_2(r_G)$, where $r_G$ is a regular representation of $G$ in the
invariant case. For metacyclic $S$ of order $\ge 4$,
we have the presentation
\begin{center}
 $S\cong G(n,m,r,l)=<a,b: a^{2^n}=1, b^{2^m}=a^{2^r}, bab^{-1}=a^l>$\\
 with $n, m\ge 1$, $0\le r\le n$, $l$ an integer $\equiv 1\pmod{2^{n-r}}$, $l^{2^m}\equiv l\pmod{2^n}$.
\end{center}
When $S$ is \textbf{metacyclic but not cyclic} with $n\ge 2$, then $s_2(r_G)=0$ if and only if $m=1$ and 
$l\equiv-1\pmod{4}$ (cf. \cite{BK2}, p. 575). 
In this case $\lambda_{1}^{G}=c_{1}^{G}=1$.

\end{rem}

\begin{cor}\label{Corollary 3.13}
 Let $G$ be a finite abelian local Galois group of $F/\bbQ_p$, where $p\ne 2$. Let $S$ be a Sylow 2-subgroup of $G$.
 \begin{enumerate}
  \item If $\text{rk}_2(S)=0$, then we have $\lambda_{1}^{G}=1$.
  \item If $\text{rk}_2(S)=1$, then 
   \begin{equation}
   \lambda_{1}^{G}=\begin{cases}
                    W(\alpha) & \text{if $|S|=2^n\ge 8$}\\
                    c_{1}^{G}\cdot W(\alpha) & \text{if $|S|\le 4$},
                   \end{cases}
\end{equation}
where $\alpha$ is a uniquely determined quadratic 
  character of $G$.
  \item If $\text{rk}_2(S)=2$, we have 
  \begin{equation}
   \lambda_{1}^{G}=\begin{cases}
                    -1 & \text{if $-1\in F^\times$ is a square element}\\
                    1 & \text{if $-1\in F^\times$ is not a square element}.
                   \end{cases}
  \end{equation}
 \end{enumerate}
\end{cor}
\begin{proof}
 This proof is straightforward from Theorem \ref{Theorem 4.3} and Corollary \ref{Lemma 3.10}.
 Here $S$ is abelian and normal because $G$ is abelian. When $\text{rk}_2(S)=0$, $G$ is of odd order, hence 
 $\lambda_{1}^{G}=1$. When $\text{rk}_2(S)=1$, $S$ is a cyclic group because $S\cong\bbZ/{2^n}\bbZ$ for some $n\ge 1$.
 Then we are in the Case 2 of Theorem \ref{Theorem 4.3}. From the Case 4 of Corollary \ref{Lemma 3.10}, 
 we can say that the case $\text{rk}_2(S)\ge 3$
 will not occur here because $p\ne 2$ and $S$ is the Galois group of a tamely ramified extension.
 
 So we are left to check the case $\text{rk}_2(S)=2$. In this case $S$ is metacyclic and contains Klein's 4-group, i.e., 
 $V\subseteq S\subseteq G$.
 Then from the properties of $\lambda$-functions and Lemma \ref{Lemma 4.6} we obtain
 \begin{equation}
  \lambda_{1}^{G}=\lambda_{1}^{V}\cdot(\lambda_{V}^{G})^4=\lambda_{1}^{V}=
  \begin{cases}
                    -1 & \text{if $-1\in F^\times$ is a square element}\\
                    1 & \text{if $-1\in F^\times$ is not a square element}.
                   \end{cases}
 \end{equation}

\end{proof}



\section{Appendix}

In the following lemma we compute an explicit formula for $\lambda_{K/F}(\psi_F)$,
where $K/F$ is a quadratic unramified extension of 
$F$. In general, for {\bf any quadratic extension} $K/F$,
we can write $\mathrm{Ind}_{K/F}1=1_F\oplus\omega_{K/F}$, where $\omega_{K/F}$ is a quadratic character of 
$F^\times$ associated to $K$ by class field theory and $1_F$ is the trivial character of $F^\times$.
Now by the definition of the $\lambda$-function we have:
\begin{equation}\label{eqn 3.333}
 \lambda_{K/F}=W(\mathrm{Ind}_{K/F}1)=W(\omega_{K/F}).
\end{equation}
So, $\lambda_{K/F}$ is the local constant of the quadratic character $\omega_{K/F}$ corresponding to $K/F$. 
\begin{lem}\label{Lemma 3.13}
 Let $K$ be the quadratic unramified extension of $F/\bbQ_p$ and let $\psi_F$ be the canonical additive character of $F$ with 
 conductor $n(\psi_F)$. Then 
 \begin{equation}
  \lambda_{K/F}(\psi_F)=(-1)^{n(\psi_F)}.
 \end{equation}
\end{lem}
\begin{proof}
When $K/F$ is the quadratic unramified extension, 
 it is easy to see that in equation (\ref{eqn 3.333})
 $\omega_{K/F}$ is an unramified character because here the ramification break $t$ is $-1$.
 Then from equation (\ref{eqn 2.3.5}) have:
 \begin{equation*}
  W(\omega_{K/F})=\omega_{K/F}(c).
 \end{equation*}
Here $\nu_F(c)=n(\psi_F)$. Therefore, from equation (\ref{eqn 3.333}) we obtain:
\begin{equation}\label{eqn 3.14}
 \lambda_{K/F}=\omega_{K/F}(\pi_F)^{n(\psi_F)}.
\end{equation}
We also know that  $\pi_F\notin N_{K/F}(K^\times)$, and hence 
$\omega_{K/F}(\pi_F)=-1$. Therefore, from equation (\ref{eqn 3.14}),
we have
\begin{equation}\label{eqn 0.4}
 \lambda_{K/F}=(-1)^{n(\psi_F)}.
\end{equation}

\end{proof}

\begin{lem}\label{Corollary 3.3.7}
The lambda function for a finite unramified extension of a non-Archimedean local field is always a sign. 
\end{lem}
 \begin{proof}
  Let $K$ be a finite unramified extension of a non-Archimedean local field $F$. 
  We know that the unramified extensions are Galois, and their corresponding
  Galois groups are cyclic. 
 Let $G=\text{Gal}(K/F)$, hence $G$ is cyclic. 
 
 When the degree of $K/F$ is odd, from Theorem \ref{General Theorem for odd case} 
 we have $\lambda_1^G=\lambda_{K/F}=1$ because $K/F$ is Galois.
 
 When the degree of $K/F$ is even, we have $\text{rk}_2(G)=1$ because $G$ is cyclic. 
  So we can write 
  $\Delta_{1}^{G}=\alpha$, where $\alpha$
  corresponds to the quadratic unramified extension. Then $\Delta_{1}^{G}(-1)=\alpha(-1)=1$, because 
  $-1$ is a norm, hence from the functional equation (\ref{eqn 2.3.23}) we have
  $$(\lambda_{1}^{G})^2=1.$$
 \end{proof}

\begin{thm}\label{Theorem 3.6}
 Let $K/F$ be a finite unramified extension with even degree and let $\psi_F$ be the canonical additive character of $F$ 
 with conductor $n(\psi_F)$. Then
 \begin{equation}
  \lambda_{K/F}=(-1)^{n(\psi_F)}.
\end{equation}
\end{thm}
\begin{proof}
When $K/F$ is a quadratic unramified extension, by Lemma \ref{Lemma 3.13}, we have $\lambda_{K/F}=(-1)^{n(\psi_F)}$.
We also know that if $K/F$ is unramified of even degree then 
 we have precisely one subextension $K'/F$ in $K/F$ such that $[K:K']=2$. Then
 \begin{center}
  $\lambda_{K/F}=\lambda_{K/K'}\cdot(\lambda_{K'/F})^2=\lambda_{K/K'}=(-1)^{n(\psi_{K'})}=(-1)^{n(\psi_F)}$,
 \end{center}
because in the unramified case the $\lambda$-function is always a sign (cf. Lemma \ref{Corollary 3.3.7}), and 
from Corollary 1 on p. 142 of \cite{AW}, $n(\psi_{K'})=n(\psi_F)$.
 
This completes the proof.

\end{proof}

In the following corollary, we show that the above Theorem \ref{Theorem 3.6} is true for any nontrivial arbitrary additive character. 
\begin{cor}\label{Corollary 3.7}
 Let $K/F$ be a finite unramified extension of even degree and let $\psi$ be any nontrivial additive character of $F$ 
 with conductor $n(\psi)$. Then
 \begin{equation}
  \lambda_{K/F}(\psi)=(-1)^{n(\psi)}.
\end{equation}
\end{cor}
\begin{proof}
 We know that any nontrivial additive character
 $\psi$ is of the form $\psi(x):=b\psi_F(x)$, for all $x\in F$, for some unique $b\in F^\times$. 
 By the definition of the conductor of an additive character of $F$, we obtain:
 \begin{center}
  $n(\psi)=n(b\psi_F)=\nu_F(b)+n(\psi_F)$.
 \end{center}
 Now let $G=\text{Gal}(K/F)$ be the Galois group of the extension $K/F$. Since $K/F$ is unramified, $G$ is {\bf cyclic}.
 Let $S$ be a Sylow $2$-subgroup of $G$. Here $S$
is nontrivial cyclic because the degree of $K/F$ is even and $G$ is cyclic. 
Then from Lemma \ref{Lemma 3.41} we have $\Delta_{1}^{G}=\Delta_{K/F}\not\equiv 1$.
Therefore, $\Delta_{K/F}(b)=(-1)^{\nu_F(b)}$ is the uniquely determined unramified quadratic character of $F^\times$.
Now from equation (\ref{eqn 3.121}) we have:
\begin{align*}
 \lambda_{K/F}(\psi)
 &=\lambda_{K/F}(b\psi_F)\\
 &=\Delta_{K/F}(b)\lambda_{K/F}(\psi_F)\\
 &=(-1)^{\nu_F(b)}\times(-1)^{n(\psi_F)},\quad \text{from Theorem $\ref{Theorem 3.6}$}\\
 &=(-1)^{\nu_F(b)+n(\psi_F)}\\
 &=(-1)^{n(\psi)}.
\end{align*}
Therefore,
when $K/F$ is an unramified extension of even degree, we have
\begin{equation}
 \lambda_{K/F}(\psi)=(-1)^{n(\psi)},
\end{equation}
where $\psi$ is any nontrivial additive character of $F$.
 
\end{proof}

In the following theorem we give an
explicit formula of $\lambda_{K/F}$, when $K/F$ is an even degree Galois extension
with odd ramification index.

\begin{thm}\label{Theorem 3.8}
 Let $K$ be an even degree Galois extension of a non-Archimedean local field $F$ of odd ramification index. 
 Let $\psi$ be a nontrivial additive character of $F$. Then 
 \begin{equation}
  \lambda_{K/F}(\psi)=(-1)^{n(\psi)}.
 \end{equation}

\end{thm}
\begin{proof}
In general,  
any extension $K/F$ of local fields has a uniquely determined
maximal subextension  $F'/F$ in $K/F$ which is unramified. Then we have 
$e_{K/F}=[K:F']$ because $e_{K/F}=e_{F'/F}\cdot e_{K/F'}=e_{K/F'}$ and $K/F'$ is a totally ramified extension. 
By the given condition, here $K/F$ is an even degree Galois extension with odd ramification index $e_{K/F}$, 
hence $K/F'$ is an odd degree Galois extension.
Now from the properties of the $\lambda$-functions and Theorem \ref{General Theorem for odd case} we have 
\begin{center}
 $\lambda_{K/F}=\lambda_{K/F'}\cdot(\lambda_{F'/F})^{e_{K/F}}=(-1)^{e_{K/F}\cdot n(\psi_F)}=(-1)^{n(\psi_F)}$,
\end{center}
because $K/F'$ is an odd degree Galois extension and $F'/F$ is an unramified extension.

\end{proof}

\begin{rem}
Finally we observe that Theorem \ref{Theorem 4.3} and Corollary \ref{Lemma 3.10} are the general results on
$\lambda_{1}^{G}=\lambda_{E/F}$, where $E/F$ is a Galois extension with Galois group $G=\rm{Gal}(E/F)$. 
And {\bf the general
results leave open} the computation of $W(\alpha)$, where $\alpha$ is a quadratic character of $G$. For
such a quadratic character we
can have three cases:
\begin{enumerate}
 \item unramified, this is Theorem \ref{Theorem 3.6}, 
 \item tamely ramified, this is Theorem 3.4.10 on p. 67 of \cite{SABT}, 
 \item wildly ramified, its explicit computation is still open.
\end{enumerate}

\end{rem}

\vspace{1cm}

\newpage

\textbf{Acknowledgements.} I would like to thank Prof E.-W. Zink, Humboldt University, Berlin
for suggesting this problem and his constant 
valuable advice and comments. I  also thank my adviser Prof. Rajat Tandon for his continuous encouragement. I
express my gratitude to Prof. Dipendra Prasad, TIFR, Bombay, for his valuable comments at the early stage of this 
project. I would also like to thank to the referees for their valuable comments for the improvement
of the paper.


\end{document}